\documentclass[11pt,oneside,reqno]{article}
\usepackage[margin=1.3in]{geometry}
\usepackage[title]{appendix}
\usepackage{titlesec}
\usepackage{amsmath,amsthm,amssymb,color,bm}
\usepackage{mathtools}
\usepackage[authoryear,round]{natbib}
\usepackage{booktabs}
\usepackage{needspace}
\usepackage{doi}

\newcommand{\wt}{\widetilde}

\usepackage{tikz}
\RequirePackage{amsfonts,mathrsfs,dsfont,mathtools,thmtools}
\RequirePackage{hyperref}
\hypersetup{colorlinks=true,citecolor=blue,urlcolor=blue,linkcolor=blue}
\usepackage[capitalize]{cleveref}
\RequirePackage{graphicx}

\crefname{equation}{}{}
\crefformat{equation}{\textup{(#2#1#3)}}
\crefrangeformat{equation}{\textup{(#3#1#4)--(#5#2#6)}}
\crefdefaultlabelformat{#2\textup{#1}#3}
\crefname{lemma}{Lemma}{Lemmas}
\crefname{conj}{Conjecture}{Conjectures}
\crefname{page}{p.}{pp.}
\usepackage[normalem]{ulem}

\usepackage{bbm}
\usepackage{enumitem}
\usetikzlibrary{arrows, automata}
\usetikzlibrary{calc}
\usetikzlibrary{positioning}
\usepackage{cancel}

\numberwithin{equation}{section}
\allowdisplaybreaks[4]

\theoremstyle{plain}
\newtheorem{theorem}{Theorem}[section]

\newtheorem{proposition}{Proposition}[section]
\newtheorem{lemma}{Lemma}[section]
\newtheorem{corollary}{Corollary}[section]

\theoremstyle{definition}

\newtheorem{example}{Example}[section]
\newtheorem{remark}{Remark}[section]

\newtheorem{conj}{Conjecture}[section]

\makeatletter

\newcount\minute
\newcount\hour
\newcount\hourMins
\def\now{%
\minute=\time%
\hour=\time \divide \hour by 60%
\hourMins=\hour \multiply\hourMins by 60%
\advance\minute by -\hourMins%
\zeroPadTwo{\the\hour}:\zeroPadTwo{\the\minute}%
}
\def\zeroPadTwo#1{\ifnum #1<10 0\fi#1}

\renewcommand{\cite}{\citet}

\def\^#1{\ifmmode {\mathaccent"705E #1} \else {\accent94 #1} \fi}
\def\~#1{\ifmmode {\mathaccent"707E #1} \else {\accent"7E #1} \fi}

\edef\-#1{\noexpand\ifmmode {\noexpand\bar{#1}} \noexpand\else \-#1\noexpand\fi}
\def\>#1{\vec{#1}}
\def\.#1{\dot{#1}}

\def\atop{\@@atop}
\def\*#1{\mathscr{#1}}

\renewcommand{\leq}{\leqslant}
\renewcommand{\le}{\leqslant}
\renewcommand{\geq}{\geqslant}
\renewcommand{\ge}{\geqslant}

\newcommand{\Prob}{\mathbb P}

\newcommand{\ind}{\mathbf 1}

\newcommand{\diag}{{\mathop{\mathrm{diag}}}}

\newcommand{\Cov}{\mathop{\mathrm{Cov}}}

\newcommand{\E}{\mathbb{E}}
\newcommand{\R}{\mathbb{R}}
\DeclareMathOperator{\tr}{tr}

\DeclareMathOperator{\dist}{dist}

\newcommand{\IR}{\mathbb{R}}

\def\be#1{\begin{equation*}#1\end{equation*}}
\def\ben#1{\begin{equation}#1\end{equation}}
\def\bes#1{\begin{equation*}\begin{split}#1\end{split}\end{equation*}}
\def\besn#1{\begin{equation}\begin{split}#1\end{split}\end{equation}}

\def\mid{\vert}

\def\beqn#1\eeqn{\begin{align}#1\end{align}}
\def\beq#1\eeq{\begin{align*}#1\end{align*}}

\usepackage{latexsym}
\usepackage{amscd}
\usepackage{epsf}

\newcommand{\mcl}[1]{\mathcal{#1}}

\renewcommand\section{\@startsection {section}{1}{\z@}%
{-3.5ex \@plus -1ex \@minus -.2ex}%
{1.3ex \@plus.2ex}%
{\center\small\sc\mathversion{bold}}}

\def\subsection#1{\@startsection {subsection}{2}{0pt}%
{-3.5ex \@plus -1ex \@minus -.2ex}%
{1ex \@plus.2ex}%
{\bf\mathversion{bold}}{#1}}

\def\subsubsection#1{\@startsection{subsubsection}{3}{0pt}%
{\medskipamount}%
{-10pt}%
{\normalsize\itshape}{\kern-2.2ex. #1.}}

\def\blfootnote{\xdef\@thefnmark{}\@footnotetext}

\makeatother

\newcommand{\pathrevisioncolor}{%
  \color{red}\hypersetup{linkcolor=red,citecolor=red,urlcolor=red}}

\makeatletter
\@namedef{path@newbib@PathKMT1975}{}
\@namedef{path@newbib@PathKMT1976}{}
\@namedef{path@newbib@PathChatterjee2012}{}
\@namedef{path@newbib@PathBG2016}{}
\@namedef{path@newbib@PathKrishnapur2020}{}
\let\path@original@bibitem\bibitem
\renewcommand{\bibitem}[2][]{%
  \par
  \ifcsname path@newbib@#2\endcsname
    \pathrevisioncolor
  \else
    \normalcolor\hypersetup{linkcolor=blue,citecolor=blue,urlcolor=blue}%
  \fi
  \path@original@bibitem[#1]{#2}}
\makeatother

\begin{document}
\author{Xiao Fang, Zi-Yao Su, Jiaxue Wang, Xiaolin Wang}
\date{\it The Chinese University of Hong Kong \\[2ex]  }
\title{Sharp Fourth-Moment Wasserstein CLT Bounds}
\maketitle
\begin{abstract}
    We prove that for the normalized sum $W=\sum_{i=1}^n X_i/\sqrt{n}$ of independent, mean-zero, $d$-dimensional random vectors with $\Cov(W) =\Sigma=\diag(\sigma_1^2,\dots, \sigma_d^2)$ (the diagonal form of the covariance matrix is assumed without loss of generality), the 2-Wasserstein distance from $W$ to $N(0,\Sigma)$ is bounded by
\[\sqrt{2\sum_{j=1}^d \min \left\{\sigma_j^2, \frac{\sum_{i=1}^n \E [|X_i'-X_i|^2(X'_{ij}-X_{ij})^2 ] }{n^2 \sigma_j^2}  \right\}},\]
where $|\cdot|$ denotes the Euclidean norm, $X_i'$ is an independent copy of $X_i$, and $X_{ij}$ ($X_{ij}'$ resp.) denotes the $j$th coordinate of the vector $X_i$ ($X_i'$ resp.).
On the one hand, it gives the optimal rate $O(\beta\sqrt{d/n})$ if $|X_i|\leq \beta$ for all $i$.
On the other hand, in the isotropic case where $\Sigma=I_d$, it gives the optimal fourth-moment bound $O(\sqrt{\sum_{i=1}^n \E |X_i|^4/n^2})$. Moreover, the coordinatewise minima suggest a tradeoff between the na\"ive coupling and the normal coupling for each component.

To prove the main result, we develop a new transportation inequality using a coupling via the F\"ollmer process with a general covariance matrix and the martingale property of its drift. The resulting bound is then controlled using a Bonis-type expansion together with Bessel's inequality.
The method and the result extend to $p$-Wasserstein bounds. As a consequence, we obtain the Hilbert space 4-Wasserstein CLT bound \( O(n^{-1/4}[\tr(Q^{1/2})]^{1/2}) \), where, in the i.i.d.\ case,
$Q=\E[|X_1'-X_1|^2(X_1'-X_1)^{\otimes2}]$ and the $O(n^{-1/4})$ rate is optimal given the dimension-free form of $Q$. We further strengthen it to a dimension-free strong approximation result with a near-optimal error rate through the KMT embedding construction.
\end{abstract}


\section{Introduction}
Multivariate central limit theorems (CLTs) with non-asymptotic error bounds have attracted great interest in the recent literature; see, for example, \cite{Bonis2020}, \cite{chernozhuokov2022improved}, and \cite{cai2026improved}.
However, even in the case of sums of independent and identically distributed (i.i.d.) random vectors, many fundamental questions remain unsolved. Compared with the one-dimensional CLT, the complication in the multivariate CLT comes from aiming for optimal dependence on dimension and on the covariance matrix.

We focus on Wasserstein bounds in the CLT.
For two probability measures $\mu$ and $\nu$ on $\mathbb{R}^d$ with finite $p$th moments, their $p$-Wasserstein distance, $p\geq 1$, is defined by
\be{
\mcl W_{p}(\mu,\nu) 
:= \inf_{\pi \in \Pi(\mu,\nu)} 
\Biggl( 
\int_{\mathbb{R}^d \times \mathbb{R}^d} 
|x-y|^p \, d\pi(x,y) 
\Biggr)^{1/p},
}
where $|\cdot|$ denotes the Euclidean norm and $\Pi(\mu,\nu)$ denotes the collection of all couplings of $\mu$ and $\nu$. In this paper, we develop bounds in terms of fourth moments that recover the optimal rates in the bounded and isotropic settings and also lead to dimension-free estimates.

The basic setting is as follows. Let $\{X_1, \dots, X_n\}$ be a sequence of independent mean-zero random vectors in $\IR^d$. Let $W_n=(X_1+\dots+X_n)/\sqrt{n}$, $\Sigma=\Cov(W_n)$ and $N(0,\Sigma)$ be the $d$-dimensional normal distribution with mean zero and covariance matrix $\Sigma$. 

In the i.i.d.\ setting and under the boundedness condition
$|X_i|\leq\beta$,
\cite{Zhai2018} obtained a 2-Wasserstein bound of order
$O(\beta\sqrt{d/n}(1+\log n))$.
This was improved by \cite{eldan2020clt} to
\ben{\label{eq:Zhai18}
\mathcal{W}_2(\mathcal{L}(W_n), N(0,\Sigma))\leq  \frac{C\beta\sqrt{d (1+\log n)}}{\sqrt{n}},
}
where $C$ is a universal constant. The bound is optimal up to the logarithmic factor, as shown by the counterexample ($\beta=\sqrt{d}$ in the example) provided in \cite[Proposition~1.2]{Zhai2018}. 

A complementary problem arises when boundedness is replaced by a fourth-moment assumption. \cite{Bonis2020} obtained a fourth-moment 2-Wasserstein bound in the isotropic case $\Sigma=I_d$, which, in the i.i.d.\ setting, gives
\ben{\label{eq:bonisbound}
\mathcal{W}_2(\mathcal{L}(W_n), N(0,I_d))\leq \frac{14d^{1/4}}{\sqrt n}
\left\|
\E\left[X_1X_1^\top|X_1|^2\right]
\right\|_{\mathrm{HS}}^{1/2},
}
where $\cdot^\top$ denotes the transpose and $\|\cdot\|_{\mathrm{HS}}$ denotes the Hilbert-Schmidt norm.
It only recovers the optimal bound \cref{eq:Zhai18} in the isotropic case. Moreover, if we only assume a finite fourth-moment condition but not the boundedness condition, the bound \cref{eq:bonisbound} can be of a worse order $O(d^{5/4}/\sqrt{n})$ than in the bounded case (see \cref{rem:bonis}). 
In the general covariance case, a direct application of \cref{eq:bonisbound} after whitening necessarily yields a bound with explicit dependence on both the dimension \(d\) and the operator norm \(\|\Sigma\|_{\mathrm{op}}\).

These comparisons motivate a fourth-moment bound that retains the covariance structure while giving the optimal dimensional order in both settings.
In this paper, we obtain the following fourth-moment bound for positive definite $\Sigma$:
\ben{\label{eq:intro3}
\mathcal{W}_2(\mathcal{L}(W_n), N(0,\Sigma))\leq\sqrt{\frac{2}{n^2}\sum_{i=1}^n \E\left[ |X_i'-X_i|^2 (X_i'-X_i)^\top\Sigma^{-1}(X_i'-X_i)\right] },
}
where $X_i'$ is an independent copy of $X_i$.
The bound \cref{eq:intro3} gives the optimal order \(O(\beta\sqrt{d/n})\) for bounded summands with general covariance, removing the logarithmic factor in \cref{eq:Zhai18} (see \cref{cor:bounded}). In the isotropic case, it yields the optimal order \(O(d/\sqrt n)\) under the typical fourth-moment condition \(\mathbb E|X_i|^4\asymp d^2\) (see \cref{cor:fourth}).  \cref{thm:fourth} gives a sharper bound that additionally incorporates coordinatewise minima with the corresponding variances.
\cref{prop:fourth} gives corresponding $p$-Wasserstein bounds for $p\geq 2$.

These estimates also lead to bounds without an inverse covariance matrix or an explicit dimensional factor. Applying a truncation and smoothing argument to the \(p=4\) case of \cref{prop:fourth}, we obtain a 4-Wasserstein bound in terms of the fourth-moment operator \(Q\) in \cref{cor:Hilbert}. The bound extends to Hilbert space, and its \(n^{-1/4}\) rate is optimal uniformly for the stated numerator. For i.i.d. vectors in \(\mathbb R^d\), a dyadic coupling construction based on the KMT scheme further yields a near-optimal maximal \(L^4\) strong approximation bound (see \cref{thm:strong}).


We now explain our proof strategy for the main bound \cref{eq:intro3}. In the literature on Wasserstein CLT bounds, a common starting point is a transportation inequality. In particular, \cite{Bonis2020} used the transportation inequality from \cite{otto2000general} and \cite{LeNoPe15}, which exploits a coupling between the random vector of interest and a Gaussian vector through the Ornstein--Uhlenbeck process. 
Recent work has started to exploit the F\"ollmer process as an alternative coupling (\cite{fang2024sharp}, \cite{fang2026martingale}). However, these results did not make use of the property that the drift in the F\"ollmer process is a martingale, a fact pointed out in \cite{lehec2013representation}.
We first exploit this martingale property to obtain a new transportation inequality in \cref{prop:transport} (without this new transportation inequality, the resulting bound will have an additional $\sqrt{\log n}$ factor). Then, following \cite{Bonis2020}, we use the score representation and the exchangeable pair approach in Stein's method (\cite{stein1986approximate}) to give an expansion of the upper bound in the transportation inequality.
Finally, we use Bessel's inequality to control the upper bound in the desired form (the application of Bessel's inequality is crucial in removing the explicit dependence on the dimension $d$ in \cref{eq:bonisbound}).
The method works for $p$-Wasserstein bounds for general $p\geq 2$ (see \cref{conj}).
We remark that \cite[Section~2]{eldan2020clt} used a similar $L^2$ version of the transportation inequality, although they did not use the score representation suitable for applying Stein's method.

The paper is organized as follows. In \cref{sec:results}, we state our main result and its corollaries and make various remarks about them. In \cref{sec:proof1}, we give the proof of our main result, making use of a new transportation inequality \cref{eq:transport-Wp}.
In \cref{sec:proof2}, we prove the other results. 

\paragraph{Statement of AI use.}  The authors formulated the research problem and conjectured an error bound similar in form to (2.1). In developing the proofs, they used GPT 6 Pro to explore possible approaches and assist with mathematical arguments. The authors critically evaluated the suggestions, revised the arguments, and verified the proofs presented in the final manuscript. They take full responsibility for the mathematical content of the paper.

\paragraph{Acknowledgements.} The investigation leading to the covariance-free Wasserstein CLT bounds in \cref{cor:Hilbert} was partly motivated by questions from Krishna Balasubramanian and Nathan Ross.
We thank Krishna Balasubramanian, Yuta Koike and Nathan Ross for helpful
discussions on the topic. Fang X. was partially supported by Hong Kong RGC GRF 14304822, 14303423, 14302124, 14304125 and a CUHK direct grant.

\section{Main results}\label{sec:results}

We first state our main result. 
Since the $p$-Wasserstein distance is invariant under rotation, we assume a diagonal covariance matrix with non-zero diagonal entries without loss of generality (see \cref{prop:fourth} for the general case).

\begin{theorem}\label{thm:fourth}
Let $X_1, \dots, X_n$ be a sequence of independent mean-zero random vectors in $\IR^d$.
Let $W_n=(X_1+\dots+X_n)/\sqrt{n}$ and $\Cov(W_n)=\Sigma=\diag(\sigma_1^2,\dots, \sigma_d^2)$. Then
\ben{\label{eq:thm1-1}
\mathcal{W}_2(\mathcal{L}(W_n), N(0,\Sigma))\leq \sqrt{2\sum_{j=1}^d \min \left\{\sigma_j^2, \frac{\sum_{i=1}^n\E \left[|X_i'-X_i|^2(X'_{ij}-X_{ij})^2\right] }{n^2 \sigma_j^2}  \right\}},
}
where $X_i'$ is an independent copy of $X_i$ and $X_{ij}$ ($X_{ij}'$ resp.) denotes the $j$th coordinate of the vector $X_i$ ($X_i'$ resp.).
\end{theorem}
The proof of \cref{thm:fourth} is given in \cref{sec:proof1}.

From \cref{thm:fourth}, we obtain the following corollary for the bounded case, which attains the optimal order.

\begin{corollary}\label{cor:bounded}
Under the conditions of \cref{thm:fourth}, assume in addition that $|X_i|\leq \beta$ for all $i\in \{1,\dots, n\}$. Then,
\ben{\label{eq:cor1-1}
\mathcal{W}_2(\mathcal{L}(W_n), N(0,\Sigma))\leq \frac{4\beta\sqrt{d}}{\sqrt{n}}.
}
\end{corollary}

\begin{proof}[Proof of \cref{cor:bounded}]
The bound \cref{eq:cor1-1} follows from \cref{eq:thm1-1} by bounding the $\min$ by the second term and using $|X_i'-X_i|\leq 2\beta$ and $\sum_{i=1}^n \E(X_{ij}'-X_{ij})^2=2n\sigma_j^2$.
\end{proof}

Next, we deduce a fourth-moment bound in the isotropic case.

\begin{corollary}\label{cor:fourth}
Under the conditions of \cref{thm:fourth}, assume in addition that $\Sigma=I_d$. Then
\ben{\label{eq:cor2-1}
\mathcal{W}_2(\mathcal{L}(W_n), N(0,I_d))\leq 4\sqrt{\frac{2\sum_{i=1}^n \E|X_i|^4}{n^2}}.
}
\end{corollary}

\begin{proof}[Proof of \cref{cor:fourth}]
    The bound \cref{eq:cor2-1} follows from \cref{eq:thm1-1} by bounding the $\min$ by the second term and using $\sigma_j^2=1$ and $\|X_i'-X_i\|_4\leq 2\|X_i\|_4$, where $\|\cdot\|_p$ denotes the $L^p$ norm of a random vector.
\end{proof}

\begin{remark}\label{rem:bonis}
Throughout this remark, assume that $X_1,\dots, X_n$ are i.i.d.
In the typical isotropic case $\|X_1\|_4\asymp\sqrt d$,
\cref{eq:cor2-1} yields the rate $O(d/\sqrt n)$,
which is optimal in general by
\cite[Proposition~1.2]{Zhai2018}.

Our fourth-moment bound also improves the dimensional
dependence of \citet{Bonis2020} for certain distributions.
Indeed, writing $M=\E[|X_1|^2 X_1X_1^\top]$,
our bound involves $(\tr M)^{1/2}$, whereas
\cref{eq:bonisbound} involves
$d^{1/4}\|M\|_{\mathrm{HS}}^{1/2}$.
Since $\tr M\le\sqrt d\,\|M\|_{\mathrm{HS}},$ our bound is always no larger, and the improvement
can be of order $d^{1/4}$. To see this, let $X_1=(U,\xi_2,\ldots,\xi_d)$,
where the $\xi_j$ are independent Rademacher variables,
independent of $U$, and $\Prob(U=\pm d)=1/(2d^2), \Prob(U=0)=1-1/d^2.$ Then $$\Cov(X_1)=I_d,\qquad M=\diag(d^2+d-1,d,\ldots,d).$$ Consequently, $\tr M\asymp d^2$ and
$\|M\|_{\mathrm{HS}}\asymp d^2$.
Our bound is therefore $O(d/\sqrt n)$,
whereas \cref{eq:bonisbound} gives
$O(d^{5/4}/\sqrt n)$. \cite{bonis2024improved} obtained error bounds involving the third-moment tensor and a radius \(h>0\) for which
\[
\mathbb{E}\!\left[(X_1'-X_1)(X_1'-X_1)^\top
\mathbf{1}_{\{|X_1'-X_1|\le h\}}\right]
\]
is positive definite. 
For the above example, \(\mathbb{E}[X_1^{\otimes3}]=0\), and the smallest admissible radius is \(h=d\) when \(d\ge2\). A direct application of his Corollary 1.1 gives a worse rate
\[
\mathcal W_2\bigl(\mathcal L(W_n),N(0,I_d)\bigr)
\le  \frac{C d^{3/2}}{\sqrt{n}}+o_d(\frac{1}{\sqrt{n}}),
\]
where $C$ is a universal constant and the small-$o$ notation $o_d$ involves a constant depending on $d$.

\cite{raivc2018multivariate} obtained a third-moment bound $\mathcal{W}_1\leq C\E|X_1|^3/\sqrt{n}$ in the isotropic case $\Sigma=I_d$ (the logarithmic factor therein can be removed according to a personal communication with Martin Rai{\v{c}}). In the typical case that $\|X_i\|_3\asymp \sqrt{d}$, his bound is of the larger order $O(d^{3/2}/\sqrt{n})$.
This is part of the reason that the recent literature on high-dimensional CLTs has started to investigate fourth-moment bounds. See related convex distance bounds in \cite{fang2024large} and \cite{cai2026improved}.
\end{remark}

The coordinatewise minima in \cref{eq:thm1-1} suggest the following result.
\begin{conj}\label{conj:uniform}
Under the conditions of \cref{thm:fourth}, there exists a
coupling $(W_n,G)$, with $G\sim N(0,\Sigma)$, such that
\[
\begin{aligned}
&\mathcal W_2^2\bigl(
\mathcal L(W_{n,j}),N(0,\sigma_j^2)
\bigr)
\le \bigl(\E|W_{n,j}-G_j|^2\bigr)\\
&\quad\le2\min\left\{
\sigma_j^2,\,
\frac1{n^2\sigma_j^2}\sum_{i=1}^n
\E\left[
|X_i'-X_i|^2(X_{ij}'-X_{ij})^2
\right]
\right\},
\end{aligned}
\]
simultaneously for all $j\in\{1,\ldots,d\}$, where $W_{n,j}$ denotes the $j$th coordinate of $W_n$.
\end{conj}
If the conjecture holds, then the minima have a natural interpretation in terms of a tradeoff between a CLT bound and a trivial bound: 
If $\sigma_j^2$ is smaller than the corresponding directional bound, it would be better to control the $j$th coordinate distance in the na\"ive way of independent coupling. This can be regarded as a more sophisticated form of truncating off small eigenvalues in bounding the Wasserstein distance. Unfortunately, the current results seem not enough to deduce the conjecture and we might need a different coupling from the standard F\"ollmer process in proving the Wasserstein CLT bounds. We leave it for future work.

Next, we turn to dimension-free Wasserstein bounds. We will prove the following result in \cref{sec:proof2} using a $p$-Wasserstein counterpart of \cref{thm:fourth} (cf. \cref{prop:fourth}).

\begin{theorem}\label{cor:Hilbert}
Under the conditions of \cref{thm:fourth}, assume that
$X_1,\ldots,X_n$ have finite fourth moments. Then
\ben{\label{eq:cor3-1}
\mathcal{W}_4(\mathcal{L}(W_n), N(0,\Sigma))\leq \frac{C[\tr (Q^{1/2}) ]^{1/2}}{n^{1/4}},
}
where $C$ is a universal constant,
\be{
		Q:=\frac1n\sum_{i=1}^n
		\E\left[|X_i'-X_i|^2(X_i'-X_i)^{\otimes2}\right]
	}
and $\otimes$ denotes the tensor product.
The same bound holds for independent, centered random elements
of a real separable Hilbert space $\ell^2$ with finite fourth moment,
with $\Sigma$ and $Q$ interpreted as the corresponding operators.
\end{theorem}

\begin{remark}\label{rem:ab}
(a) The $n^{-1/4}$ rate in \cref{cor:Hilbert} is optimal given the
form of the numerator, even for bounded summands and for the 2-Wasserstein distance. 
Indeed, let $d=m$, $n=m^2$, and let $X_1,\ldots,X_n$ be
independent copies of $X$ with
\[
\Prob(X=e_j)=\Prob(X=-e_j)=\frac1{2m^2},
\quad j=1,\ldots,m,
\qquad
\Prob(X=0)=1-\frac1m.
\]
Then $|X|\le1$, $\Sigma=m^{-2}I_m$, and a direct calculation gives
\[
\tr(Q^{1/2})=\sqrt{2+\frac2m+\frac4{m^2}}\asymp1.
\]
Since $W_n\in m^{-1}\mathbb Z^m$, for $G\sim N(0,\Sigma)$,
\[
\mathcal W_2\bigl(\mathcal L(W_n),N(0,\Sigma)\bigr)
\ge \sqrt{\sum_{j=1}^m\E\dist(G_j,m^{-1}\mathbb Z)^2}
=\sqrt{\frac{\E\dist(Z,\mathbb Z)^2}{m}}
\asymp n^{-1/4},
\]
where $Z\sim N(0,1)$.

 
(b) One may wonder if we could replace the numerator of the upper bound in \cref{eq:cor3-1} with a simpler quantity involving only $\E |X_i'-X_i|^4$. However, this is impossible. 
In the infinite-dimensional Hilbert-space setting, a finite
fourth moment alone does not guarantee any universal rate of convergence.
In \cref{sec:proof2}, for every positive sequence $a_n\downarrow0$, we will give an example of
a fixed mean-zero random vector $X$ in $\ell^2$, with $|X|=1$ almost surely,
such that, for independent copies $X_1,X_2,\ldots$ of $X$,
\[
 W_n=\frac{1}{\sqrt n}\sum_{i=1}^nX_i,
 \qquad \Sigma=\operatorname{Cov}(X),
 \qquad d_n=\mathcal W_2\bigl(\mathcal L(W_n),N(0,\Sigma)\bigr),
\]
we have
\[
 d_n\longrightarrow0,
 \qquad \frac{d_n}{a_n}\longrightarrow\infty.
\]

(c) As discussed above, the error bound in \cref{eq:cor3-1} is already optimal for $\mathcal{W}_2$ with the displayed numerator. It is perhaps surprising that we can raise it to the stronger $\mathcal{W}_4$. This is associated with the variance-free nature of the upper bound. For comparison, classical \(O(n^{-1/2})\) bounds for \(\mathcal W_p\) are obtained under a finite \((p+2)\)th moment assumption; see \cite[Theorem~1.1]{bobkov2018berry}. For \(p>4\), a general bound must involve additional \(p\)th-moment information, while the bounded example in (a) rules out a faster uniform rate within the present formulation. We therefore focus on \(p=4\), the largest Wasserstein order that can be controlled by the fourth-moment quantity in \cref{eq:cor3-1} alone.
\end{remark}

In our final result, we address the natural question of strong approximations. We obtain the following result by combining our new $p$-Wasserstein bounds with the Koml\'os--Major--Tusn\'ady (KMT) embedding construction (see the paragraph before the proof of \cref{thm:strong} in \cref{sec:proof2} for a discussion on the related literature).

\begin{theorem}\label{thm:strong}
Let $X_1,\ldots,X_n$ be i.i.d. mean-zero random vectors
in $\R^d$ with finite fourth moment and covariance matrix
$\Sigma$. There is a coupling of these vectors with mutually
independent $G_1,\ldots,G_n\sim N(0,\Sigma)$ such that
\ben{\label{eq:thmstrong}
\Big(\E \max_{1\leq k\leq n} \Big|\frac{1}{\sqrt{n}}\sum_{i=1}^k (X_i-G_i)\Big|^4\Big)^{1/4}\leq \frac{C[\tr (Q^{1/2})]^{1/2}}{n^{1/4}}(1+\log n),
}
where $C$ is a universal constant and 
\be{
		Q:=
		\E\left[|X_1'-X_1|^2(X_1'-X_1)^{\otimes2}\right].
	}
\end{theorem}

\begin{remark}
\cref{thm:strong} provides the first dimension-free strong approximation result in the literature. Moreover, the error rate in terms of $n$ is optimal up to the logarithmic factor (see \cref{exam:strong} at the end of \cref{sec:proof2}). It might be possible to remove the logarithmic factor and extend the result to the independent but non-identically distributed setting by adapting the arguments in \cite{gotze2010rates} and \cite{GZ2011}, but we have chosen not to pursue overly technical extensions.
\end{remark}


\section{Proof of the main result}\label{sec:proof1}

We first state a general $p$-Wasserstein
bound for $p\geq 2$. 
Let $\Sigma^\dagger$ denote the Moore--Penrose
pseudoinverse of $\Sigma$, where the inverse of
a zero eigenvalue is defined to be zero.


\begin{proposition}\label{prop:fourth}
Let $p\geq 2$.
Let $X_1,\ldots,X_n$ be independent mean-zero
random vectors in $\R^d$ with finite $p$th moments, and let
$W_n=(X_1+\cdots+X_n)/\sqrt n$ with
$\Cov(W_n)=\Sigma$. Then
\besn{\label{conj}
\mathcal W_p(\mathcal L(W_n),N(0,\Sigma))
\leq& \frac{C_p}{\sqrt n}
\Bigg\{
\left(
\frac1n\sum_{i=1}^n
\E\left[
|X_i'-X_i|^2
(X_i'-X_i)^\top\Sigma^\dagger(X_i'-X_i)
\right]
\right)^{1/2}\\
&\qquad\,+
\left(
\frac1n\sum_{i=1}^n
\E\left[
|X_i'-X_i|^p
(X_i'-X_i)^\top\Sigma^\dagger(X_i'-X_i)
\right]
\right)^{1/p}
\Bigg\},
}
where $C_p$ is a constant depending only on $p$ and $C_2$ can be taken as $1/\sqrt{2}$.

\end{proposition}

\cref{thm:fourth} follows easily from \cref{prop:fourth}.

\begin{proof}[Proof of \cref{thm:fourth}]
For any $J\subseteq\{1,\ldots,d\}$, applying
\cref{prop:fourth} to the coordinates in $J$ gives
\[
\mathcal W_2^2
\bigl(\mathcal L((W_n)_J),N(0,\Sigma_J)\bigr)
\le
\frac{2}{n^2}\sum_{j\in J}\frac{1}{\sigma_j^2}
\sum_{i=1}^n
\E\left[|X_i'-X_i|^2(X_{ij}'-X_{ij})^2\right].
\]
Let $((W_n)_J,G_J)$ be an optimal coupling, where
$G_J\sim N(0,\Sigma_J)$, and extend this coupling to
the full vector $W_n$ while preserving its law.
Let $G_{J^c}\sim N(0,\Sigma_{J^c})$ be independent
of $(W_n,G_J)$. Since $\Sigma$ is diagonal,
$G=(G_J,G_{J^c})\sim N(0,\Sigma)$. Hence,
\begin{align*}
\mathcal W_2^2\bigl(\mathcal L(W_n),N(0,\Sigma)\bigr)\le& \E|(W_n)_J-G_J|^2
      +\E|(W_n)_{J^c}-G_{J^c}|^2\\=&\mathcal W_2^2
\bigl(\mathcal L((W_n)_J),N(0,\Sigma_J)\bigr)
+2\sum_{j\notin J}\sigma_j^2\\
\leq&
\frac{2}{n^2}\sum_{j\in J}\frac{1}{\sigma_j^2}
\sum_{i=1}^n
\E\left[
|X_i'-X_i|^2(X_{ij}'-X_{ij})^2
\right]
+2\sum_{j\notin J}\sigma_j^2.
\end{align*}
Minimizing over $J$ yields
\[
\mathcal W_2^2\bigl(\mathcal L(W_n),N(0,\Sigma)\bigr)
\le
2\sum_{j=1}^d
\min\left\{
\sigma_j^2,\,
\frac{1}{n^2\sigma_j^2}
\sum_{i=1}^n
\E\left[|X_i'-X_i|^2(X_{ij}'-X_{ij})^2\right]
\right\}.
\]
\end{proof}

\begin{proof}[Proof of \cref{prop:fourth}]

\noindent\textbf{Step 1: $L^p$ transportation inequality.}
In this step, we obtain a new transportation inequality, \cref{prop:transport}, using the F\"ollmer process. See related results in \cite{lehec2013representation}, \cite{eldan2020clt}, \cite{fang2024sharp}, \cite{fang2026central}, \cite{koike2026connections}, \cite{fang2026martingale}. In particular, we follow the proof of \cite[Lemma~5.1]{fang2026martingale} with a crucial modification using the martingale property of the drift of the F\"ollmer process.

\begin{proposition}\label{prop:transport}
Let $p\geq 2$, and let $\mu$ be a centered probability measure on $\IR^d$ with finite $p$th moment. Let $\Sigma$ be a symmetric, positive definite matrix. We can construct two random vectors $W\sim \mu$ and $Z\sim N(0, I_d)$ defined on a common probability space such that
\begin{equation}\label{eq:transport-Wp}
\|W-\Sigma^{1/2} Z\|_p
\le K_p\sqrt{\int_0^\infty
(1-e^{-2t})\|\rho_t(F_t)\|_p^2\,dt},
\end{equation}
where $Z_2\sim N(0, I_d)$ is independent of $W$, $K_p$ is a constant depending only on $p$, $K_2$ can be taken as $2$,
and, for $t>0$,
\ben{\label{eq:rhot}
\rho_t(F_t)
=
e^{-t}\E\left[
W-\frac{\Sigma^{1/2}Z_2}{\sqrt{e^{2t}-1}}
\,\middle|\,F_t
\right],
\qquad
F_t=e^{-t}W+\sqrt{1-e^{-2t}}\,\Sigma^{1/2}Z_2.
}
\end{proposition}

\begin{proof}[Proof of \cref{prop:transport}]
In this proof, we use $K_p$ to denote constants that depend only on $p$ and may differ in each appearance.
We follow \cite[Proof of Lemma~5.1]{fang2026martingale} and construct
\be{
W-\Sigma^{1/2} Z=\int_0^1 \frac{\Sigma^{1/2}\nabla \log f_u(X_u/\sqrt{u})}{\sqrt{u}} du,
}
where $X=(X_s)_{s\in [0,1]}$ is the F\"ollmer process associated with the law of $\bar W:=\Sigma^{-1/2} W$ and $f_u$ is the density of $X[u]:=\sqrt{u}\bar W+\sqrt{1-u}Z_2$ with respect to the standard Gaussian measure $N(0, I_d)$. 

In \cite[Proof of Lemma~5.1]{fang2026martingale}, the Minkowski inequality was used to obtain
\be{
\|W-\Sigma^{1/2} Z\|_p\leq \int_0^1 \frac{\|\Sigma^{1/2}\nabla \log f_u(X_u/\sqrt{u})\|_p}{\sqrt{u}} du.
}
Here, we make use of the martingale property of $(v_u)_{0\leq u<1}$ where 
\be{
v_u=\Sigma^{1/2}\nabla \log f_u(X_u/\sqrt{u})/\sqrt{u};
}
see \cite[Lemma~11]{lehec2013representation} and \cite[Proposition~2.9]{koike2026connections}. 

Let $v_0=\E W=0$ and $[v]$ be the scalar quadratic variation of $v$. 
For $T<1$, integration by parts and the
Burkholder--Davis--Gundy inequalities give
\begin{align*}
\left\|\int_0^T v_u\,du\right\|_p^2
&=\left\|\int_0^T(T-u)\,dv_u\right\|_p^2\le K_p\left\|
\int_0^T(T-u)^2\,d[v]_u
\right\|_{p/2}\\
&=2K_p\left\|
\int_0^T(T-u)[v]_u\,du
\right\|_{p/2}\le K_p\int_0^T(T-u)\|v_u\|_p^2\,du,
\end{align*}
where the last inequality follows from Minkowski's inequality, the reverse Burkholder--Davis--Gundy inequality, and Doob's maximal inequality.

Since $W-\Sigma^{1/2}Z=\int_0^1 v_u\,du$,
we have $\int_0^T v_u\,du\to W-\Sigma^{1/2}Z$
almost surely as $T\uparrow1$.
Fatou's lemma and monotone convergence then yield
\begin{align*}
\|W-\Sigma^{1/2} Z\|_p^2
&\le K_p\int_0^1(1-u)\|v_u\|_p^2\,du.
\end{align*}
As in \cite[Proof of Lemma~5.1]{fang2026martingale}, we observe that $X_u/\sqrt{u}$ has the same law as $X[u]$ and the score representation for Gaussian convolutions in \cite[Lemma~2]{Bonis2020} gives
\be{
\nabla \log f_u(X[u])=\sqrt{u} \E \left[ \Sigma^{-1/2} W -\sqrt{\frac{u}{1-u}}Z_2  \Big\vert \Sigma^{1/2} X[u]  \right].
}
The proof is complete by the same change of variables 
\be{
u=e^{-2t},\quad t\in [0,\infty)
}
as in \cite[Proof of Lemma~5.1]{fang2026martingale}.
The constant in the case $p=2$ is obtained using the martingale covariance identity
\be{
\E|\int_0^T v_u du|^2=2\int_0^T(T-u)\E|v_u|^2du.
}
\end{proof}

\noindent\textbf{Step 2: Score representation and Bonis expansion.}
In this step, we use exchangeable pairs to represent $\rho_t$ and bound $\|\rho_t(F_t)\|_p^2$ as in \cite{Bonis2020} and \cite{fang2023p}. The result is summarized in the following lemma.

\begin{lemma}\label{lem:bonisexpansion}
Let $p\geq 2$, and
let $(W,W')$ be an exchangeable pair of random vectors (i.e., $\mathcal{L}(W, W')=\mathcal{L}(W', W)$)
in $\R^d$ with finite $p$th moments and
$\Cov(W)=\Sigma$ positive definite.
Let $D:=W'-W$ and suppose that
\[
\E[D\mid W]=-\lambda W
\]
for some $\lambda>0$.
For $t>0$, let $D_t=D\ind_{\{|\Sigma^{-1/2}D|
\leq\sqrt{e^{2t}-1}\}}.$ 
For a random tensor $T$, write
$\|T\|_{p,\mathrm{HS}}=(\mathbb E\|T\|_{\mathrm{HS}}^p)^{1/p}$.
Then, for every $t>0$,
\begin{equation}\label{eq:bonisexpansion}
\begin{aligned}
 e^{2t}\|\rho_t(F_t)\|_p^2
 &\le \frac{1}{\lambda^2}
       \|\mathbb E[D-D_t\mid W]\|_p^2\\
 &\quad+\frac{p-1}{e^{2t}-1}
       \left\|\left(\frac{\mathbb E[D_tD_t^\top\mid W]}{2\lambda}
                  -\Sigma\right)\Sigma^{-1/2}\right\|_{p,\mathrm{HS}}^2\\
 &\quad+\frac{1}{4\lambda^2}\sum_{k=3}^{\infty}
       \frac{\bigl(1-k(k-1)/4\bigr)^2(p-1)^k}
            {k!(e^{2t}-1)^k}
       \left\|\mathbb E\!\left[
          D_t\otimes(\Sigma^{-1/2}D_t)^{\otimes k}\mid W
       \right]\right\|_{p,\mathrm{HS}}^2.
\end{aligned}
\end{equation}
\end{lemma}

\begin{proof}[Proof of \cref{lem:bonisexpansion}]
Let $Z_2\sim N(0,I_d)$ be independent of $(W, W')$ as in \cref{eq:rhot}. Write $Z=Z_2$ in this proof.
Applying the calculation in \citet[(6.6)--(6.12)]{fang2023p} to
$(\Sigma^{-1/2}W,\Sigma^{-1/2}W')$
and
multiplying the resulting vector identity by $\Sigma^{1/2}$ gives
\begin{equation*}
   e^t\rho_t(F_t)=\mathbb E[R_t\mid F_t],
\end{equation*}
where
\[
\begin{aligned}
 R_t
 &=-\frac1\lambda\mathbb E[D-D_t\mid W]\\
 &\quad+\frac1{\sqrt{e^{2t}-1}}
       \left(\frac{\mathbb E[D_tD_t^\top\mid W]}{2\lambda}-\Sigma\right)
                         \Sigma^{-1/2}Z\\
 &\quad+\frac1{2\lambda}\sum_{k=3}^{\infty}
       \frac{1-\frac{k(k-1)}4}{k!(e^{2t}-1)^{k/2}}
       \mathbb E\left[
        D_t\left\langle H_k(Z),(\Sigma^{-1/2}D_t)^{\otimes k}\right\rangle
                  \,\middle|\,W,Z\right]
\end{aligned}
\]
and $H_k(z)=(-1)^k e^{|z|^2/2}\nabla^k e^{-|z|^2/2}$. Therefore,
\ben{\label{eq:condjensen}
\|e^t \rho_t(F_t)\|_p^2\leq \|R_t\|_p^2=\|(\E[|R_t|^p \vert W])^{2/p} \|_{p/2}.
}
Conditionally on $W$, the terms in $R_t$ belong to Gaussian Hermite
chaoses of orders $0,1,3,4,\ldots$. We then apply
\citet[Lemma~3]{Bonis2020} to bound 
\besn{\label{eq:bonislemma3}
(\E[|R_t|^p \vert W])^{2/p}\leq &\frac{1}{\lambda^{2}}|\mathbb E[D-D_t\mid W]|^2\\
& +\frac{p-1}{e^{2t}-1}
 \left\|\left(\frac{\mathbb E[D_tD_t^\top\mid W]}{2\lambda}
                  -\Sigma\right)\Sigma^{-1/2}\right\|_{\mathrm{HS}}^2\\
&+ \frac{1}{4\lambda^2}\sum_{k=3}^{\infty}
       \frac{\bigl(1-k(k-1)/4\bigr)^2(p-1)^k}
            {k!(e^{2t}-1)^k}
       \left\|\mathbb E\!\left[
          D_t\otimes(\Sigma^{-1/2}D_t)^{\otimes k}\mid W
       \right]\right\|_{\mathrm{HS}}^2
}
Combining \cref{eq:condjensen} and \cref{eq:bonislemma3} proves \eqref{eq:bonisexpansion}.
\end{proof}

\noindent\textbf{Step 3: Bounding by centering, Rosenthal's inequality,
and Bessel's inequality.}
To bound the right-hand side of \cref{eq:bonisexpansion}, we follow the independent-sum centering argument in
\citet[Sections~5 and~6]{Bonis2020}. The centered terms are bounded by
the multidimensional version of Rosenthal inequality as in \cite{Bonis2020}. The difference is that we use the Bessel
inequality (see \cref{eq:sum-bessel}) for the deterministic tensor means.

We assume without loss of generality that $\Sigma$ is positive definite and that the right-hand side of \eqref{conj} is finite.
Below, we use $C_p$ to denote positive constants that depend only  on $p$ and may
change from line to line.

We use the standard construction of exchangeable pairs for sums of independent random vectors. 
Let $(X_1',\ldots,X_n')$ be an independent copy of $(X_1,\ldots,X_n)$, and let $I$ be uniform on $\{1,\ldots,n\}$ and independent of both vectors. Set
\[
 W=W_n,\qquad W'=W+D,\qquad
 D=\frac{X_I'-X_I}{\sqrt n},\qquad \lambda=\frac1n.
\]
It can be computed that $\mathbb E[D\mid W]=-\lambda W$. Write
\[
 \Delta_i=X_i'-X_i,\qquad
 \Delta_{i,t}=\Delta_i
 \mathbf1_{\{|\Sigma^{-1/2}\Delta_i|^2\le n(e^{2t}-1)\}},
 \qquad D_t=\frac{\Delta_{I,t}}{\sqrt n}.
\]
We will apply the Bessel inequality in the following form.
Note that
\[
 \mathbb E[(\Sigma^{-1/2}\Delta_I)(\Sigma^{-1/2}\Delta_I)^\top]=2I_d.
\]
Thus $(\Sigma^{-1/2}\Delta_I)_j/\sqrt2$, $j=1,\ldots,d$, form an
orthonormal family in $L^2$. For arbitrary square-integrable random
tensors $U_i$ of the same order across $i$, each depending on $(X_i,X_i')$, Bessel's
inequality gives
\begin{equation}\label{eq:sum-bessel}
 \left\|\sum_{i=1}^n
       \mathbb E[U_i\otimes\Sigma^{-1/2}\Delta_i]
                   \right\|_{\mathrm{HS}}^2
       \le 2n\sum_{i=1}^n\mathbb E\|U_i\|_{\mathrm{HS}}^2.
\end{equation}
Indeed, for each scalar coordinate $U_{i,\alpha}$,
\[
\begin{aligned}
 \frac1{2n^2}\sum_{j=1}^d
   \left(\sum_{i=1}^n
     \mathbb E[U_{i,\alpha}(\Sigma^{-1/2}\Delta_i)_j]\right)^2
 &=\sum_{j=1}^d\left|
     \mathbb E\left[U_{I,\alpha}
             \frac{(\Sigma^{-1/2}\Delta_I)_j}{\sqrt2}\right]\right|^2\\
 &\le\mathbb E|U_{I,\alpha}|^2
   =\frac1n\sum_{i=1}^n\mathbb E|U_{i,\alpha}|^2.
\end{aligned}
\]
Summing over $\alpha$ proves \eqref{eq:sum-bessel}.

We first consider the first term on the right-hand side of \cref{eq:bonisexpansion}.
By symmetry, $\mathbb E[\Delta_i-\Delta_{i,t}]=0$.
Conditioning first on $(X_1,\ldots,X_n)$ and applying the multidimensional version of
Rosenthal's inequality of \citet[Theorem~5.2]{Pinelis1994}, as in
\citet[Section~6]{Bonis2020}, gives
\begin{equation}\label{eq:firstline}
\begin{aligned}
 \frac1{\lambda^2}\|\mathbb E[D-D_t\mid W]\|_p^2
 &\le\frac1n
       \left\|\sum_{i=1}^n
             \mathbb E[\Delta_i-\Delta_{i,t}\mid X_i]\right\|_p^2\\
 &\le\frac{C_p}{n}\left\{
       \sum_{i=1}^n\mathbb E|\Delta_i-\Delta_{i,t}|^2
       +\left(\sum_{i=1}^n\mathbb E|\Delta_i-\Delta_{i,t}|^p
          \right)^{2/p}\right\}\\
 &\le\frac{C_p}{n}\left\{
       \sum_{i=1}^n\mathbb E|\Delta_i-\Delta_{i,t}|^2
       +\left(\frac{1}{n(e^{2t}-1)}\sum_{i=1}^n\mathbb E[|\Delta_i|^p |\Sigma^{-1/2}\Delta_i|^2]
          \right)^{2/p}\right\},
\end{aligned}
\end{equation}
where we used truncation in the last inequality.

To prepare for the following arguments, we also observe that the same Rosenthal inequality gives, for every $k\ge1$,
\besn{\label{eq:rosenthalk}
 &\left\|\sum_{i=1}^n\left\{
       \mathbb E[\Delta_{i,t}\otimes(\Sigma^{-1/2}\Delta_{i,t})^{\otimes k}
                      \mid X_i]
       -\mathbb E[\Delta_{i,t}\otimes(\Sigma^{-1/2}\Delta_{i,t})^{\otimes k}]
                       \right\}\right\|_{p,\mathrm{HS}}^2\\
 &\quad\le C_p\left\{
       \sum_{i=1}^n\mathbb E[
           |\Delta_{i,t}|^2|\Sigma^{-1/2}\Delta_{i,t}|^{2k}]
       +\left(\sum_{i=1}^n\mathbb E[
           |\Delta_{i,t}|^p|\Sigma^{-1/2}\Delta_{i,t}|^{kp}]
        \right)^{2/p}\right\}.
}

Next, we consider the second term on the right-hand side of \cref{eq:bonisexpansion}. Since $\sum_{i=1}^n\mathbb E[\Delta_i\Delta_i^\top]=2n\Sigma$,
centering the conditional truncated matrices leaves the deterministic
bias
\[
 -\frac1{2n}\sum_{i=1}^n\mathbb E[
           (\Delta_i-\Delta_{i,t})\otimes\Sigma^{-1/2}\Delta_i].
\]
Taking $U_i=\Delta_i-\Delta_{i,t}$ in \cref{eq:sum-bessel} gives the truncation-bias bound
\[
 \left\|\sum_{i=1}^n\mathbb E[
      (\Delta_i-\Delta_{i,t})\otimes\Sigma^{-1/2}\Delta_i]
                  \right\|_{\mathrm{HS}}^2
       \le 2n\sum_{i=1}^n\mathbb E|\Delta_i-\Delta_{i,t}|^2.
\]
Combining the truncation-bias bound with the centered-sum estimate \cref{eq:rosenthalk} with $k=1$ yields
\begin{equation}\label{eq:secondline}
\begin{aligned}
 &\left\|\left(\frac{\mathbb E[D_tD_t^\top\mid W]}{2\lambda}
                      -\Sigma\right)\Sigma^{-1/2}
              \right\|_{p,\mathrm{HS}}^2\\
 &\quad\le\frac{C_p}{n^2}\left\{
       \sum_{i=1}^n\mathbb E[
             |\Delta_{i,t}|^2|\Sigma^{-1/2}\Delta_{i,t}|^2]
       +n\sum_{i=1}^n\mathbb E|\Delta_i-\Delta_{i,t}|^2\right.\\
 &\hspace{8em}\left.
       +\left(\sum_{i=1}^n\mathbb E[
             |\Delta_{i,t}|^p|\Sigma^{-1/2}\Delta_{i,t}|^p]
         \right)^{2/p}\right\}\\
&\quad\le\frac{C_p}{n^2}\left\{
       \sum_{i=1}^n\mathbb E[
             |\Delta_{i,t}|^2|\Sigma^{-1/2}\Delta_{i,t}|^2]
       +n\sum_{i=1}^n\mathbb E|\Delta_i-\Delta_{i,t}|^2\right.\\
 &\hspace{8em}\left.
       +\left([n(e^{2t}-1)]^{p/2-1}\sum_{i=1}^n\mathbb E[
             |\Delta_{i}|^p|\Sigma^{-1/2}\Delta_{i}|^2]
         \right)^{2/p}\right\},
\end{aligned}
\end{equation}
where we again used truncation in the last inequality.

Finally, we consider the last line on the right-hand side of \cref{eq:bonisexpansion}. 
Taking
$U_i=\Delta_{i,t}\otimes(\Sigma^{-1/2}\Delta_{i,t})^{\otimes(k-1)}$ in \cref{eq:sum-bessel}
gives, for $k\ge3$,
\[
 \left\|\sum_{i=1}^n\mathbb E[
      \Delta_{i,t}\otimes(\Sigma^{-1/2}\Delta_{i,t})^{\otimes k}]
                  \right\|_{\mathrm{HS}}^2\le 2n\sum_{i=1}^n\mathbb E[
                |\Delta_{i,t}|^2|\Sigma^{-1/2}\Delta_{i,t}|^{2k-2}].
\]
Together with \cref{eq:rosenthalk}, we obtain, for $k\ge3$,
\besn{\label{eq:3rdtemp}
 &\left\|\mathbb E\left[\sum_{i=1}^n
       \Delta_{i,t}\otimes(\Sigma^{-1/2}\Delta_{i,t})^{\otimes k}
                      \,\middle|\,W\right]\right\|_{p,\mathrm{HS}}^2\\
 &\quad\le C_p\left\{
       \sum_{i=1}^n\mathbb E\left[
        |\Delta_{i,t}|^2\left(
           |\Sigma^{-1/2}\Delta_{i,t}|^{2k}
             +2n|\Sigma^{-1/2}\Delta_{i,t}|^{2k-2}\right)\right]
       \right.\\
 &\hspace{6em}\left.
       +\left(\sum_{i=1}^n\mathbb E[
          |\Delta_{i,t}|^p|\Sigma^{-1/2}\Delta_{i,t}|^{kp}]
         \right)^{2/p}\right\}.
}

Elementary truncation gives, for $k\ge1$,
\[
\begin{aligned}
 \sum_{i=1}^n\mathbb E[
       |\Delta_{i,t}|^p|\Sigma^{-1/2}\Delta_{i,t}|^{kp}]
 &\le[n(e^{2t}-1)]^{kp/2-1}
       \sum_{i=1}^n\mathbb E[|\Delta_i|^p|\Sigma^{-1/2}\Delta_i|^2].
\end{aligned}
\]
For $k\ge3$, the powers $2k$ and $2k-2$ in \cref{eq:3rdtemp}
are similarly bounded by
$[n(e^{2t}-1)]^{k-2}|\Sigma^{-1/2}\Delta_{i,t}|^4$ and
$[n(e^{2t}-1)]^{k-3}|\Sigma^{-1/2}\Delta_{i,t}|^4$, respectively.
Using
\[
 \mathbb E[D_t\otimes(\Sigma^{-1/2}D_t)^{\otimes k}\mid W]
 =\frac1{n^{(k+3)/2}}\mathbb E\left[
       \sum_{i=1}^n\Delta_{i,t}\otimes
                (\Sigma^{-1/2}\Delta_{i,t})^{\otimes k}
                 \,\middle|\,W\right],
\]
we therefore obtain
\begin{equation}\label{eq:thirdline}
\begin{aligned}
 &\frac1{4\lambda^2}\sum_{k=3}^{\infty}
       \frac{\left(1-\frac{k(k-1)}4\right)^2(p-1)^k}{k!(e^{2t}-1)^k}
       \left\|\mathbb E[
         D_t\otimes(\Sigma^{-1/2}D_t)^{\otimes k}\mid W]
                    \right\|_{p,\mathrm{HS}}^2\\
 &\quad\le\frac{C_p}{n^3}
       \left(\frac1{(e^{2t}-1)^2}+\frac1{(e^{2t}-1)^3}\right)
       \sum_{i=1}^n\mathbb E[
              |\Delta_{i,t}|^2|\Sigma^{-1/2}\Delta_{i,t}|^4]\\
 &\qquad+\frac{C_p}{n^{1+2/p}(e^{2t}-1)^{2/p}}
       \left(\sum_{i=1}^n\mathbb E[
               |\Delta_i|^p|\Sigma^{-1/2}\Delta_i|^2]\right)^{2/p}.
\end{aligned}
\end{equation}
Together with \cref{eq:firstline,eq:secondline,eq:thirdline}, \eqref{eq:bonisexpansion} implies
\besn{\label{eq:pointwisebound}
 e^{2t}\|\rho_t(F_t)\|_p^2
 &\le \frac{C_p}{n}\left(1+\frac1{e^{2t}-1}\right)
         \sum_{i=1}^n\mathbb E|\Delta_i-\Delta_{i,t}|^2\\
 &\quad+\frac{C_p}{n^2(e^{2t}-1)}
         \sum_{i=1}^n\mathbb E[
             |\Delta_{i,t}|^2|\Sigma^{-1/2}\Delta_{i,t}|^2]\\
 &\quad+\frac{C_p}{n^3}
       \left(\frac1{(e^{2t}-1)^2}+\frac1{(e^{2t}-1)^3}\right)
         \sum_{i=1}^n\mathbb E[
             |\Delta_{i,t}|^2|\Sigma^{-1/2}\Delta_{i,t}|^4]\\
 &\quad+\frac{C_p}{n^{1+2/p}(e^{2t}-1)^{2/p}}
         \left(\sum_{i=1}^n\mathbb E[
             |\Delta_i|^p|\Sigma^{-1/2}\Delta_i|^2]\right)^{2/p}.
}

\noindent\textbf{Step 4: Integrating the bound.}
It remains to integrate the pointwise bound in \cref{eq:pointwisebound} to control
\be{
\int_0^\infty(1-e^{-2t})\|\rho_t(F_t)\|_p^2\,dt.
}
We will use the change of
variables $u=e^{2t}-1$ and
\[
 (1-e^{-2t})e^{-2t}\,dt=\frac{u}{2(1+u)^3}\,du.
\]

For the first term in \cref{eq:pointwisebound},
we have
\begin{align*}
 &\frac{C_p}{n}\int_0^\infty
 (1-e^{-2t})e^{-2t}
 \left(1+\frac{1}{e^{2t}-1}\right)
 \sum_{i=1}^n\mathbb E|\Delta_i-\Delta_{i,t}|^2\,dt\\
=&\frac{C_p}{2n}\sum_{i=1}^n
 \mathbb E\left[
 |\Delta_i|^2\int_0^{|\Sigma^{-1/2}\Delta_i|^2/n}
 \frac{du}{(1+u)^2}\right]\\
=&\frac{C_p}{2n^2}\sum_{i=1}^n
 \mathbb E\left[
 \frac{|\Delta_i|^2|\Sigma^{-1/2}\Delta_i|^2}
 {1+|\Sigma^{-1/2}\Delta_i|^2/n}\right]\\
\quad&\le\frac{C_p}{2n^2}\sum_{i=1}^n
 \mathbb E\left[|\Delta_i|^2|\Sigma^{-1/2}\Delta_i|^2\right].
\end{align*}

For the second term in \cref{eq:pointwisebound},
we have
\begin{align*}
 &\frac{C_p}{n^2}\int_0^\infty
 \frac{(1-e^{-2t})e^{-2t}}{e^{2t}-1}
 \sum_{i=1}^n
 \mathbb E\left[|\Delta_{i,t}|^2|\Sigma^{-1/2}\Delta_{i,t}|^2\right]\,dt\\
 &\quad=\frac{C_p}{2n^2}\sum_{i=1}^n
 \mathbb E\left[
 |\Delta_i|^2|\Sigma^{-1/2}\Delta_i|^2
 \int_{|\Sigma^{-1/2}\Delta_i|^2/n}^\infty
 \frac{du}{(1+u)^3}\right]\\
 &\quad=\frac{C_p}{4n^2}\sum_{i=1}^n
 \mathbb E\left[
 \frac{|\Delta_i|^2|\Sigma^{-1/2}\Delta_i|^2}
 {(1+|\Sigma^{-1/2}\Delta_i|^2/n)^2}\right]\\
 &\quad\le\frac{C_p}{4n^2}\sum_{i=1}^n
 \mathbb E\left[|\Delta_i|^2|\Sigma^{-1/2}\Delta_i|^2\right].
\end{align*}

For the third term in \cref{eq:pointwisebound},
since
\[
 \frac{u}{(1+u)^3}\left(\frac1{u^2}+\frac1{u^3}\right)
 =\frac{1}{u^2(1+u)^2},
\]
we obtain
\begin{align*}
 &\frac{C_p}{n^3}\int_0^\infty
 (1-e^{-2t})e^{-2t}
 \left(\frac{1}{(e^{2t}-1)^2}+\frac{1}{(e^{2t}-1)^3}\right)\\
 &\hspace{28mm}\times\sum_{i=1}^n
 \mathbb E\left[|\Delta_{i,t}|^2|\Sigma^{-1/2}\Delta_{i,t}|^4\right]\,dt\\
 &\quad=\frac{C_p}{2n^3}\sum_{i=1}^n
 \mathbb E\left[
 |\Delta_i|^2|\Sigma^{-1/2}\Delta_i|^4
 \int_{|\Sigma^{-1/2}\Delta_i|^2/n}^\infty
 \frac{du}{u^2(1+u)^2}\right]\\
 &\quad\le\frac{C_p}{2n^3}\sum_{i=1}^n
 \mathbb E\left[
 |\Delta_i|^2|\Sigma^{-1/2}\Delta_i|^4
 \int_{|\Sigma^{-1/2}\Delta_i|^2/n}^\infty
 \frac{du}{u^2}\right]\\
 &\quad=\frac{C_p}{2n^2}\sum_{i=1}^n
 \mathbb E\left[|\Delta_i|^2|\Sigma^{-1/2}\Delta_i|^2\right].
\end{align*}

For $p\ge2$ and $u>0$, concavity gives
\[
 u^{1-2/p}\le \frac2p+\left(1-\frac2p\right)u.
\]
Consequently, for the fourth term in \cref{eq:pointwisebound}, we have
\begin{align*}
 \int_0^\infty\frac{u^{1-2/p}}{(1+u)^3}\,du
 &\le\frac2p\int_0^\infty\frac{du}{(1+u)^3}
 +\left(1-\frac2p\right)\int_0^\infty\frac{u\,du}{(1+u)^3}
 =\frac12.
\end{align*}
It follows that
\begin{align*}
 &\frac{C_p}{n^{1+2/p}}
 \left(\sum_{i=1}^n
 \mathbb E[|\Delta_i|^p|\Sigma^{-1/2}\Delta_i|^2]\right)^{2/p}
 \int_0^\infty
 \frac{(1-e^{-2t})e^{-2t}}{(e^{2t}-1)^{2/p}}\,dt\\
 &\quad=\frac{C_p}{2n^{1+2/p}}
 \left(\sum_{i=1}^n
 \mathbb E[|\Delta_i|^p|\Sigma^{-1/2}\Delta_i|^2]\right)^{2/p}
 \int_0^\infty\frac{u^{1-2/p}}{(1+u)^3}\,du\\
 &\quad\le\frac{C_p}{4n^{1+2/p}}
 \left(\sum_{i=1}^n
 \mathbb E[|\Delta_i|^p|\Sigma^{-1/2}\Delta_i|^2]\right)^{2/p}.
\end{align*}
Summing these four estimates, enlarging $C_p$, and using
$|\Sigma^{-1/2}\Delta_i|^2=\Delta_i^\top\Sigma^{-1}\Delta_i$, we conclude that
\begin{equation}\label{eq:score-bound-Wp}
\begin{aligned}
 \int_0^\infty(1-e^{-2t})\|\rho_t(F_t)\|_p^2\,dt
 \le&\frac{C_p}{n^2}\sum_{i=1}^n
 \mathbb E[|\Delta_i|^2\Delta_i^\top\Sigma^{-1}\Delta_i]\\
 &+\frac{C_p}{n^{1+2/p}}
 \left(\sum_{i=1}^n
 \mathbb E[|\Delta_i|^p\Delta_i^\top\Sigma^{-1}\Delta_i]\right)^{2/p}.
\end{aligned}
\end{equation}
Combining \cref{eq:score-bound-Wp} with \eqref{eq:transport-Wp} yields
\[
\begin{aligned}
 \mathcal W_p\bigl(\mathcal L(W_n),N(0,\Sigma)\bigr)
 \le\frac{C_p}{\sqrt n}\biggl\{&
  \left(\frac1n\sum_{i=1}^n
       \mathbb E[|\Delta_i|^2\Delta_i^\top\Sigma^{-1}\Delta_i]
       \right)^{1/2}\\
 &+\left(\frac1n\sum_{i=1}^n
       \mathbb E[|\Delta_i|^p\Delta_i^\top\Sigma^{-1}\Delta_i]
       \right)^{1/p}\biggr\}.
\end{aligned}
\]
This proves \eqref{conj} for a constant depending only on the fixed
$p$.

\paragraph{The constant when $p=2$.}
For $p=2$, we use exact variance decompositions instead of Rosenthal's
inequality. Applying
variance decompositions and \eqref{eq:sum-bessel} in
\eqref{eq:bonisexpansion}, and using
\[
 \sum_{k=3}^{\infty}\frac{\left(1-\frac{k(k-1)}4\right)^2}{k!}
                =\frac{15e-34}{16}<\frac12,
\]
we obtain
\[
\begin{aligned}
 e^{2t}\mathbb E|\rho_t(F_t)|^2
 &\le\frac1n\left(1+\frac1{2(e^{2t}-1)}\right)
            \sum_{i=1}^n\mathbb E|\Delta_i-\Delta_{i,t}|^2\\
 &\quad+\frac1{4n^2(e^{2t}-1)}\sum_{i=1}^n
          \mathbb E[|\Delta_{i,t}|^2|\Sigma^{-1/2}\Delta_{i,t}|^2]\\
 &\quad+\frac1{8n^3}
          \left(\frac1{(e^{2t}-1)^2}+\frac2{(e^{2t}-1)^3}\right)
           \sum_{i=1}^n
          \mathbb E[|\Delta_{i,t}|^2|\Sigma^{-1/2}\Delta_{i,t}|^4].
\end{aligned}
\]
The transportation constant $K_2=2$ in \cref{eq:transport-Wp} gives the integration weight
$2u\,du/(1+u)^3$. For $r>0$,
\[
\begin{aligned}
 \int_0^r\frac{2u+1}{(1+u)^3}\,du&\le r,\\
 \frac12\int_r^\infty\frac{du}{(1+u)^3}
       &=\frac1{4(1+r)^2}\le\frac14,\\
 \frac14\int_r^\infty\frac{u+2}{u^2(1+u)^3}\,du
       &\le\frac12\int_r^\infty\frac{du}{u^2}=\frac1{2r}.
\end{aligned}
\]
Applying these inequalities with $r=|\Sigma^{-1/2}\Delta_i|^2/n$ gives
\[
\begin{aligned}
 \mathcal W_2^2\bigl(\mathcal L(W_n),N(0,\Sigma)\bigr)
 &\le4\int_0^\infty(1-e^{-2t})\mathbb E|\rho_t(F_t)|^2\,dt\\
 &\le\frac{1+1/4+1/2}{n^2}\sum_{i=1}^n
          \mathbb E[|\Delta_i|^2\Delta_i^\top\Sigma^{-1}\Delta_i]\\
 &\le\frac2{n^2}\sum_{i=1}^n
          \mathbb E[|\Delta_i|^2\Delta_i^\top\Sigma^{-1}\Delta_i].
\end{aligned}
\]
\end{proof}

\section{Proofs of other results}\label{sec:proof2}

We first prove the dimension-free bound stated in \cref{cor:Hilbert}.
\begin{proof}[Proof of \cref{cor:Hilbert}]
In this proof, we use $C$ to denote universal constants, which may differ from line to line.
We may assume that
$0<\tr(Q^{1/2})<\infty$, since otherwise the bound is trivial.

We first consider the finite-dimensional case.
Let
$\Sigma_i=\Cov(X_i)$. By a straightforward computation using independence, we obtain
\besn{\label{eq:expandQ}
&\E\left[|X_i'-X_i|^2(X_i'-X_i)^{\otimes2}\right]=2\E\left[|X_i|^2X_i^{\otimes2}\right]
 +2(\tr\Sigma_i)\Sigma_i+4\Sigma_i^2.
}
This implies $Q\succeq 4\Sigma_i^2/n$ and any vector $v$ in the kernel of $Q$ satisfies $\langle X_i, v\rangle =0$.
Therefore, each $X_i$ lies in the range of $Q$ almost surely.
We work on the subspace of the range of $Q$, on which $Q$ is positive definite and all the inverse of matrices below are well defined.
Taking traces in \cref{eq:expandQ} yields
\ben{\label{eq:boundbyeps}
\frac1{n^2}\sum_{i=1}^n\E|X_i|^4
\le\frac{\tr Q}{2n}
\le\frac{\{\tr(Q^{1/2})\}^2}{2n}.
}

Because the $\mathcal{W}_4$ bound in \cref{prop:fourth} involves the sixth moments, we will first do a truncation. Let 
\be{
\varepsilon=n^{-1/4}\{\tr(Q^{1/2})\}^{1/2}
}
and
\be{
\wt W_n=\frac{1}{\sqrt{n}}\sum_{i=1}^n\wt X_i,\quad 
\wt X_i=U_i-\E U_i,\quad U_i=X_i\mathbf1_{\{|X_i|\le\sqrt n\,\varepsilon\}}.
}
Let $\wt X_i'$ and $U_i'$ be defined from $X_i'$ in the same way as $\wt X_i$ and $U_i$.
Let 
\be{
\wt Q=\frac1n\sum_{i=1}^n
\E\left[|\wt X_i'-\wt X_i|^2
(\wt X_i'-\wt X_i)^{\otimes2}\right].
}
In the next lemma, we bound $\wt Q$ by a constant multiple of $Q$.
\begin{lemma}\label{lem:qtildebound}
We have
\be{
\wt Q\preceq 4Q.
}
\end{lemma}
\begin{proof}[Proof of \cref{lem:qtildebound}]
The inequality
\bes{
|u-v|^2(u-v)^{\otimes2}\preceq& 2(|u|^2+|v|^2)(u-v)^{\otimes2}\\
\preceq& 2(|u|^2+|v|^2)[(u-v)^{\otimes2}+(u+v)^{\otimes2}]\\
\preceq&4(|u|^2+|v|^2)(u^{\otimes2}+v^{\otimes2})
}
and $\wt X_i'-\wt X_i=U_i'-U_i$ imply
\bes{
&\E\left[|\wt X_i'-\wt X_i|^2
(\wt X_i'-\wt X_i)^{\otimes2}\right]\\
 \preceq &4\E(|U_i|^2+|U_i'|^2)(U_i^{\otimes 2}+U_i'^{\otimes 2})\\
=&8\E\left[|U_i|^2U_i^{\otimes2}\right]
 +8\E|U_i|^2\,\E[U_i^{\otimes2}]\\
=&8\E\left[|X_i|^2X_i^{\otimes2}\mathbf1_{\{|X_i|\leq \sqrt{n} \varepsilon\}}\right]
 +8\E|X_i\mathbf1_{\{|X_i|\leq \sqrt{n} \varepsilon\}}|^2\,\E[X_i^{\otimes2}\mathbf1_{\{|X_i|\leq \sqrt{n} \varepsilon\}}]\\ 
\preceq&8\E\left[|X_i|^2X_i^{\otimes2}\right]
 +8(\tr\Sigma_i)\Sigma_i,
}
where we recall that $\Sigma_i=\Cov(X_i)$. 
Together with \cref{eq:expandQ}, we obtain
\be{
\E\left[|\wt X_i'-\wt X_i|^2
(\wt X_i'-\wt X_i)^{\otimes2}\right]\preceq4\E\left[|X_i'-X_i|^2
(X_i'-X_i)^{\otimes2}\right]
}
and $\wt Q\preceq4Q$. 
\end{proof}
 
In the next lemma, we compare $\wt W_n$ with the corresponding Gaussian vector.
\begin{lemma}\label{lem:truncatedcomp}
Let 
\be{
\wt \Sigma=\Cov(\wt W_n).
}
We have
\be{
\mathcal{W}_4(\mathcal{L}(\wt W_n), N(0,\wt \Sigma))\leq C\varepsilon.
}
\end{lemma}
\begin{proof}[Proof of \cref{lem:truncatedcomp}]
Because the error bound in \cref{prop:fourth} involves the inverse of the covariance matrix, we will do a smoothing to get rid of it.
Let
\be{
A=n^{-1/2}Q^{1/2},\quad M=\wt \Sigma+A, \quad \delta_i=\frac{\wt X_i'-\wt X_i}{\sqrt n}.
}
Take $Z\sim N(0,A)$ as the sum of $K$
independent $N(0,A/K)$ vectors, jointly independent of the
$\wt X_i$'s. Applying \cref{prop:fourth} with $p=4$ and letting $K\to\infty$ give
\bes{
&\mathcal W_4\bigl(\mathcal L(\wt W_n+Z),N(0,M)\bigr)\\
\le& C\left(\sum_{i=1}^n
\E[|\delta_i|^2\delta_i^\top M^{-1}\delta_i]\right)^{1/2}
+C\left(\sum_{i=1}^n
\E[|\delta_i|^4\delta_i^\top M^{-1}\delta_i]\right)^{1/4}.
}
Since $M\succeq A$, we have $M^{-1}\preceq A^{-1}$ and, together with \cref{lem:qtildebound}, we obtain
\bes{
\sum_{i=1}^n\E[|\delta_i|^2\delta_i^\top M^{-1}\delta_i]
&=\frac1n\tr(M^{-1}\wt Q)\le\frac4n\tr(A^{-1}Q)
=\frac4{\sqrt n}\tr(Q^{1/2})
=4\varepsilon^2.
}
Furthermore, $|\delta_i|\le2\varepsilon$, so
\bes{
\sum_{i=1}^n\E[|\delta_i|^4\delta_i^\top M^{-1}\delta_i]
&\le4\varepsilon^2\sum_{i=1}^n
\E[|\delta_i|^2\delta_i^\top M^{-1}\delta_i]\le16\varepsilon^4.
}
It follows that
\be{
\mathcal W_4\bigl(\mathcal L(\wt W_n+Z),N(0,M)\bigr)
\le C\varepsilon.
}
Also, for any positive semidefinite matrix $A$ and $Z_A\sim N(0,A)$,
\ben{\label{eq:gaussfourth}
\E|Z_A|^4=(\tr A)^2+2\tr(A^2)\le3(\tr A)^2.
}
Therefore,
$\|Z\|_4\le 3^{1/4}\varepsilon$. Removing the smoothing Gaussian from
both laws and using the triangle inequality now gives the desired bound
\be{
\mathcal W_4\bigl(\mathcal L(\wt W_n),N(0,\wt\Sigma)\bigr)
\le C\varepsilon.
}
\end{proof}

We are only left to bound $\mathcal W_4\bigl(\mathcal L(\wt W_n), \mathcal{L}(W_n)\bigr)$ and $\mathcal W_4\bigl(N(0,\wt\Sigma) ,N(0,\Sigma)\bigr)$.

Denoting $V_i=X_i\mathbf 1_{\{|X_i|>\sqrt{n} \varepsilon\}}-\E X_i\mathbf 1_{\{|X_i|>\sqrt{n} \varepsilon\}}$ which has mean zero and is independent across $i$, we have
\bes{
\E|W_n-\wt W_n|^4&=\frac{1}{n^2}\E |\sum_{i=1}^n V_i|^4=\frac{1}{n^2} \E\Big(\sum_{i=1}^n |V_i|^2 +2\sum_{i<j} \langle V_i, V_j \rangle \Big)^2\\
&=\frac{1}{n^2} \Big(\sum_{i=1}^n \E |V_i|^4 +2 \sum_{i<j} \E |V_i|^2|V_j|^2 + 4 \sum_{i<j} \E\langle V_i, V_j \rangle^2 \Big)\\
&\leq \frac{1}{n^2} \Big(\sum_{i=1}^n \E |V_i|^4 +6 \sum_{i<j} \E |V_i|^2|V_j|^2 \Big)\\
&\leq \frac{1}{n^2} \Big(\sum_{i=1}^n \E |V_i|^4 +3 (\sum_{i=1}^n \E |V_i|^2)^2 \Big).
}
Using
$\E|Y-\E Y|^4\le16\E|Y|^4$, we obtain
\besn{\label{eq:truncationerror}
\E|W_n-\wt W_n|^4
\le\frac{16}{n^2}\sum_{i=1}^n\E|X_i|^4+3\left\{\frac1n\sum_{i=1}^n
\E\left[|X_i|^2\mathbf1_{\{|X_i|>\sqrt n\,\varepsilon\}}\right]
\right\}^2
\le C\varepsilon^4,
}
where we used \cref{eq:boundbyeps} in the last inequality.
For comparing $N(0,\wt\Sigma)$ and $N(0,\Sigma)$, observe that
\be{
\Sigma-\wt\Sigma
=\frac1n\sum_{i=1}^n\left\{
\E\left[X_i^{\otimes2}\mathbf1_{\{|X_i|>\sqrt n\,\varepsilon\}}\right]
+(\E U_i)^{\otimes2}\right\}\succeq0.
}
Since $\E U_i=-\E[X_i\mathbf1_{\{|X_i|>\sqrt n\,\varepsilon\}}]$,
Jensen's inequality gives
$\tr(\Sigma-\wt\Sigma)\le\varepsilon^2$.
Consequently, coupling the two Gaussian laws by adding an
independent $N(0,\Sigma-\wt\Sigma)$ vector and using \cref{eq:gaussfourth} give
\ben{\label{eq:truncationerror2}
\mathcal W_4\bigl(N(0,\wt\Sigma),N(0,\Sigma)\bigr)
\le3^{1/4}\{\tr(\Sigma-\wt\Sigma)\}^{1/2}
\le3^{1/4}\varepsilon.
}
Combining \cref{lem:truncatedcomp} with \cref{eq:truncationerror} and \cref{eq:truncationerror2} proves
\be{
\mathcal W_4\bigl(\mathcal L(W_n),N(0,\Sigma)\bigr)
\le C\varepsilon
=\frac{C\{\tr(Q^{1/2})\}^{1/2}}{n^{1/4}}.
}

\paragraph{Hilbert space CLT bound.} Let $P_m$ be the projection onto the span of the first $m$ vectors of an orthonormal basis, with
$P_m\to I$ strongly. The averaged fourth-moment operator of
$P_mX_i$ is
\[
Q_m=\frac1n\sum_{i=1}^n\E\left[
|P_m(X_i'-X_i)|^2(P_m(X_i'-X_i))^{\otimes2}
\right]\preceq P_mQP_m.
\]
Using the operator monotonicity of the square root in the first inequality and the ideal
property of the trace norm in the second inequality below, we have
\[
\begin{aligned}
\tr(Q_m^{1/2})
&\le \tr((P_mQP_m)^{1/2})=\|Q^{1/2}P_m\|_1
\le \|Q^{1/2}\|_1
=\tr(Q^{1/2}),
\end{aligned}
\]
where $\|\cdot\|_1$ denotes the trace norm.
Applying the finite-dimensional result to $P_mX_i$
therefore gives the same bound uniformly in $m$.
Finally, $P_mW_n\to W_n$ and
$N(0,P_m\Sigma P_m)\to N(0,\Sigma)$ in $\mathcal W_4$.
Letting $m\to\infty$ completes the proof.
\end{proof}

Next, we give an example supporting the claim in \cref{rem:ab}(b).

\begin{example}
Recall the setting of \cref{rem:ab}(b). The example below is adapted from 
\cite[Theorem~7 and Example~5]{GZ2011}, who considered strong
Gaussian approximation.

Let $(e_j)_{j\ge1}$ be the standard orthonormal basis of $\ell^2$, and let
$(p_j)_{j\ge1}$ be positive numbers with $\sum_jp_j=1$. Define
\[
 \mathbb P(X=e_j)=\mathbb P(X=-e_j)=\frac{p_j}{2},
 \qquad j\ge1.
\]
Then $|X|=1$, $\mathbb EX=0$, and $\Sigma e_j=p_je_j$.
Since $W_{n,j}\in n^{-1/2}\mathbb Z$, every coupling of $W_n$ and
$G\sim N(0,\Sigma)$ satisfies a coordinatewise lattice lower bound.
Consequently, with $Z\sim N(0,1)$,
\ben{\label{eq:dn2}
 \begin{aligned}
 d_n^2
 &\ge \sum_{j\ge1}\mathbb E\operatorname{dist}
       \bigl(\sqrt{p_j}Z,n^{-1/2}\mathbb Z\bigr)^2\ge c_0\sum_{j:\,np_j\le1}p_j,
 \qquad
 c_0:=\mathbb E\bigl[Z^2\mathbf1_{\{|Z|\le1/2\}}\bigr]>0.
 \end{aligned}
}
Indeed, if $np_j\le1$ and $|Z|\le1/2$, the nearest lattice point to
$\sqrt{p_j}Z$ is zero. Hence the variance carried by coordinates with
$p_j\le1/n$ cannot yet be approximated at a smaller relative cost.

To obtain the asserted arbitrary slowness, choose strictly increasing
integers $N_k$ such that
\[
 N_k\ge4^k,\qquad a_{N_k}\le2^{-k},\qquad k\ge1.
\]
Partition the coordinates into disjoint finite blocks $B_k$, with
\[
 |B_k|=m_k:=\bigl\lceil2^{-k}N_k\bigr\rceil,
 \qquad p_j=\frac{2^{-k}}{m_k}\quad(j\in B_k).
\]
The resulting probabilities sum to one, and $p_j\le1/N_k$ on $B_k$.
If $N_k\le n<N_{k+1}$, every coordinate in $B_{k+1},B_{k+2},\ldots$
has $np_j\le1$. Therefore, from \cref{eq:dn2},
\[
 d_n^2\ge c_0\sum_{\ell\ge k+1}2^{-\ell}=c_0\,2^{-k},
 \qquad
 \frac{d_n}{a_n}\ge\sqrt{c_0}\,2^{k/2}\longrightarrow\infty.
\]
Nevertheless, $d_n\to0$. To see this, let
$P_m$ be the projection onto the first $m$ coordinates. The finite-dimensional
CLT and convergence of second moments give
$$\mathcal W_2(\mathcal L(P_mW_n),\mathcal L(P_mG))\to0$$ for each fixed
$m$, whereas
\[
 \mathbb E|(I-P_m)W_n|^2
 =\mathbb E|(I-P_m)G|^2
 =\sum_{j>m}p_j
\]
uniformly in $n$. The triangle inequality, first letting $n\to\infty$
and then $m\to\infty$, proves the claim.
\end{example}

Next, we prove the strong approximation result, \cref{thm:strong}. The basic coupling argument is from the KMT embedding theorem (see  \cite[Section~3]{komlos1975approximation} and \cite[Section~4]{komlos1976approximation}), and the difference is the way we handle the coupling errors. See \cite{Z2002} and the references therein for a survey of KMT approximation results for sums of independent, non-identically distributed random variables and random vectors.
For extensions of the KMT method to the dependent case, see, for example, \cite{chatterjee2012new}, \cite{bhattacharjee2016strong}, \cite{dimitrov2021kmt} and \cite{ye2025computable}.

\begin{proof}[Proof of \cref{thm:strong}]
We first describe a dyadic construction based on the KMT method. Next, we use \cref{cor:Hilbert} to bound the fourth moments of the
coupling errors introduced by random halving.
Finally, we express the partial sum errors in terms of these
coupling errors and sum the resulting bounds over the dyadic levels.

In this proof, we use $C$ to denote universal constants, which may differ at each appearance.
We may assume
that $0<\tr(Q^{1/2})<\infty$; the other cases are
immediate. In particular, conditional Jensen's inequality gives
$\E|X_1|^4\le\E|X_1'-X_1|^4=\tr Q<\infty$.
We first suppose that $n=2^J\ge2$ for an integer $J$. General $n\geq 2$ will be considered at the end of this proof.

\medskip
\noindent\textbf{Step 1: Dyadic construction.}
We first describe the dyadic construction used in the KMT embedding. Start with a labeled i.i.d. sample  
$\widetilde X_1,\ldots,\widetilde X_n$. For the block $I=I_L\cup I_R$, where $I_L=\{1,\dots, 2^{J-1}\}$ and $I_R=\{2^{J-1}+1,\dots, 2^J\}$, we rearrange the $\widetilde X_i$'s in the following way. Starting with the pair $\{\wt X_1, \wt X_2\}$, we flip a fair coin to determine which $\wt X_i$ goes to position 1 (the first position in $I_L$), with the other going to position $2^{J-1}+1$ (the first position in $I_R$). Mathematically, this is achieved by using a Rademacher random variable $\eta_1$ with $\mathbb P(\eta_1=\pm 1)=1/2$: if $\eta_1=1$, $\wt X_1$ goes to position 1, and if $\eta_1=-1$, $\wt X_2$ goes to position 1. Repeat this procedure for all the other consecutive pairs of the $\wt X_i$'s using independent $\eta_2, \dots, \eta_{2^{J-1}}$. Finally, repeat the above procedure for each resulting sub-block $I_L, I_R, \dots$ until there is only one element left in each block.
At each split, conditionally on the original sample and all variables generated before that split, the sign vector has the product Rademacher law. It follows that the label rearrangement is independent of the original sample; hence the resulting sequence, denoted by $\{X_1,\dots, X_n\}$, is still i.i.d.\ with the original marginal distribution.

Next, we create a Gaussian process with the desired marginal distribution, coupled with the above procedure. First generate $Z_n\sim N(0,n\Sigma)$ (to be coupled in \cref{eq:strong-root-coupling} with $\sum_{i=1}^n\wt X_i$). Next, given $Z_n$ and $\wt X_1,\dots,\wt X_n$, jointly generate the signs $\eta_1,\dots,\eta_{2^{J-1}}$ and the bridge $B_n$, with conditional marginal laws equal to the product Rademacher law and $N(0,n\Sigma)$, respectively, so that $B_n$ is close to
\[
\sum_{i=1}^{2^{J-1}}\eta_i(\wt X_{2i-1}-\wt X_{2i})
\]
in the sense of \cref{eq:strong-bridge-fourth}. 
Since the conditional distribution of $B_n$ given $Z_n$ and the $\wt X_i$'s is the fixed law $N(0,n\Sigma)$, $B_n$ is independent of $\{Z_n,\wt X_1,\dots,\wt X_n\}$.
Define $Z_{n/2}:=(Z_n+B_n)/2$. Finally, repeat the above procedure for $(I_L,Z_{n/2})$, $(I_R,Z_n-Z_{n/2})$, $\dots$ until there is only one element left in each block. Because the bridge vectors are mutually independent and independent of $Z_n$, the resulting $Z$ process has the same distribution as the partial sum process $\sum_{i=1}^kG_i$, $k=1,\dots,n$, where the $G_i$'s are i.i.d.\ with distribution $N(0,\Sigma)$. The difference between the two partial sum processes $\sum_{i=1}^kX_i$ and $\sum_{i=1}^kG_i$ satisfies the standard dyadic interpolation identity; see \cref{eq:strong-hat-expansion}.

\medskip
\noindent\textbf{Step 2: Fourth moments of the coupling errors.}
By \cref{cor:Hilbert}, the total sum $\sum_{i=1}^n \wt X_i$ can be coupled with
$Z_{n}\sim N(0,n\Sigma)$ so that
\begin{equation}\label{eq:strong-root-coupling}
 \left\|\sum_{i=1}^n\widetilde X_i-Z_{n}\right\|_4
 \le Cn^{1/4}\{\tr(Q^{1/2})\}^{1/2}.
\end{equation}

Now consider the coupling for a dyadic block $I$ with $2m$ elements, denoted before reordering by $\{x_1, x_1', \dots, x_m, x_m'\}$. These $x$'s are fixed for now. We will take expectations with respect to their randomness at the end of this step. After reordering by independent Rademacher random variables $\eta^I_1, \dots, \eta^I_m$ (they are independent across $I$), the difference between the left and right half-block sums is
\be{
D_I:=\sum_{i=1}^m \eta_i^I (x_i-x_i').
}
In the next lemma, we will prove that the Gaussian vector $B_I\sim N(0, 2m\Sigma)$ can be constructed to satisfy the following fourth-moment error bound.
\begin{lemma}\label{lem:eIbound}
Write $\eta_i:=\eta_i^I$, $a_i:=x_i-x_i'$ and
\[
 \Gamma=\sum_{i=1}^m a_i^{\otimes2},
 \qquad
 R=\sum_{i=1}^m |a_i|^2a_i^{\otimes2}.
\]
We can construct $B_I$ such that
\begin{equation}\label{eq:strong-conditional-fourth}
\begin{aligned}
 \E[|D_I-B_I|^4]
 &=\mathcal W_4^4\bigl(
   \mathcal L(D_I),N(0,2m\Sigma)\bigr)\le C\left[
   \{\tr(R^{1/2})\}^2+\|\Gamma-2m\Sigma\|_1^2
 \right],
\end{aligned}
\end{equation}
where $\|\cdot\|_1$ denotes the trace norm (the sum of the singular values).
\end{lemma}

\begin{proof}[Proof of \cref{lem:eIbound}]

We first compare $\mathcal L(D_I)$ with
$N(0,\Gamma)$.
Let \((\eta_1',\ldots,\eta_m')\) be an independent copy of \((\eta_1,\ldots,\eta_m)\).
To apply \cref{cor:Hilbert}, we first compute that
\be{
\frac1m\sum_{i=1}^m
\E\left[
|(\eta_i'-\eta_i)a_i|^2
\{(\eta_i'-\eta_i)a_i\}^{\otimes2}
\right]=\frac8mR.
}
From \cref{cor:Hilbert}, we obtain
\begin{equation}\label{eq:strong-rademacher-fourth}
\begin{aligned}
\mathcal W_4\bigl(\mathcal L(D_I),N(0,\Gamma)\bigr)
&\le C\sqrt m\,m^{-1/4}
\left\{\tr\left[\left(\frac8mR\right)^{1/2}\right]\right\}^{1/2}\le C\{\tr(R^{1/2})\}^{1/2}.
\end{aligned}
\end{equation}

Next, we compare $N(0, \Gamma)$ with $N(0, 2m\Sigma)$. Add an independent $N(0,(\Gamma-2m\Sigma)_-)$ vector
to an $N(0,\Gamma)$ vector, and an independent
$N(0,(\Gamma-2m\Sigma)_+)$ vector to an $N(0,2m\Sigma)$ vector.
Here $(\cdot)_+$ and $(\cdot)_-$ denote the positive and
negative parts of a symmetric matrix.
The two sums have the same Gaussian law because
\[
 \Gamma+(\Gamma-2m\Sigma)_-
 =2m\Sigma+(\Gamma-2m\Sigma)_+.
\]
We also have
\[
 \tr((\Gamma-2m\Sigma)_-)
 +\tr((\Gamma-2m\Sigma)_+)
 =\|\Gamma-2m\Sigma\|_1.
\]
The Gaussian fourth-moment bound \cref{eq:gaussfourth} and the triangle inequality
therefore give
\begin{equation}\label{eq:strong-gaussian-fourth}
 \mathcal W_4\bigl(N(0,\Gamma),N(0,2m\Sigma)\bigr)
 \le C\|\Gamma-2m\Sigma\|_1^{1/2}.
\end{equation}
Combining \eqref{eq:strong-rademacher-fourth} and
\eqref{eq:strong-gaussian-fourth} gives \cref{eq:strong-conditional-fourth}.
\end{proof}

We now average the bound \cref{eq:strong-conditional-fourth} with respect to the randomness of the $x$'s in the block $I$.
As argued in Step~1, the label assignments are independent
of the original sample. Thus the $a_i$ are unconditionally
independent copies of $X_1-X_1'$. Since
$\E a_i^{\otimes2}=2\Sigma$ and
$(a_i^{\otimes2})^2=|a_i|^2a_i^{\otimes2}$, independence between the $a_i$'s gives
\ben{\label{eq:momentidentity}
 \E R=mQ,
 \qquad
 \E(\Gamma-2m\Sigma)^2
 =m(Q-4\Sigma^2)\preceq mQ.
}
It can be checked that the range of $Q$ contains the $a_i$'s almost surely and contains the range of $\Sigma$.
As in the proof of \cref{cor:Hilbert}, we work on the subspace of the range of $Q$, on which $Q$ is positive definite and all the inverse of matrices below are well defined.
The Cauchy--Schwarz inequality and the 
moment identities \cref{eq:momentidentity} give
\begin{align*}
 \E\{\tr(R^{1/2})\}^2
 &\le \tr(Q^{1/2})\tr\bigl((\E R)Q^{-1/2}\bigr)
 =m\{\tr(Q^{1/2})\}^2,\\
 \E\|\Gamma-2m\Sigma\|_1^2
 &\le \tr(Q^{1/2})
       \tr\bigl(\E[(\Gamma-2m\Sigma)^2]Q^{-1/2}\bigr)\le m\{\tr(Q^{1/2})\}^2.
\end{align*}
Taking expectations in \eqref{eq:strong-conditional-fourth}
therefore gives
\begin{equation}\label{eq:strong-bridge-fourth}
 \E|D_I-B_I|^4\le Cm\{\tr(Q^{1/2})\}^2,
 \qquad |I|=2m.
\end{equation}

\medskip
\noindent\textbf{Step 3: Summation over dyadic levels.}
Let $e_0=\sum_{i=1}^n X_i-Z_n$ and $e_I:=D_I-B_I$ for each dyadic block.
By the standard dyadic interpolation identity, the partial sum errors admit the decomposition
\begin{equation}\label{eq:strong-hat-expansion}
 \sum_{i=1}^k(X_i-G_i)
 =\frac{k}{n}e_0+\frac12\sum_I\phi_I(k)e_I,
 \qquad k=0,\ldots,n,
\end{equation}
where the sum is over all dyadic blocks in $(0,n]$ with
$|I|\ge2$, and, for each dyadic block $I=(\ell,r]$,
$\phi_I$ is the continuous hat function supported on $[\ell,r]$,
linear on each half, equal to one at the midpoint and zero
at the endpoints.


We next use the fourth-moment bound from Step~2 to control
the largest coupling error at each level.
A level with block size $2m$ contains $n/(2m)$ blocks.
By \eqref{eq:strong-bridge-fourth},
\begin{equation}\label{eq:strong-level-fourth}
\begin{aligned}
&\left\|\max_{|I|=2m}|e_I|\right\|_4\le \left(\sum_{|I|=2m}\E|e_I|^4\right)^{1/4}\\
 &\le \left(
       \frac{n}{2m}\,Cm\{\tr(Q^{1/2})\}^2
       \right)^{1/4}\le Cn^{1/4}\{\tr(Q^{1/2})\}^{1/2}.
\end{aligned}
\end{equation}
For each $k$, at most one hat function per level is nonzero,
and $0\le\phi_I(k)\le1$. Thus the contribution of each level
in \eqref{eq:strong-hat-expansion} is bounded by half the
largest $|e_I|$ on that level.
Taking the maximum over $k$ and applying Minkowski's inequality,
\eqref{eq:strong-root-coupling}, and
\eqref{eq:strong-level-fourth}, we obtain
\[
\begin{aligned}
 \left\|\max_{1\le k\le n}
 \left|\sum_{i=1}^k(X_i-G_i)\right|\right\|_4
 &\le \|e_0\|_4
 +\frac12\sum_{j=1}^J
 \left\|\max_{|I|=2^j}|e_I|\right\|_4\\
&\le Cn^{1/4}\{\tr(Q^{1/2})\}^{1/2}(1+J).
\end{aligned}
\]
Since $J=\log_2 n$, this proves
\cref{eq:thmstrong} for dyadic $n$.

\paragraph{General $n\ge2$.}
For general $n\ge2$, choose $N=2^{\lceil\log_2 n\rceil}$, so that $n\le N<2n$. Applying the dyadic construction to $N$ i.i.d. increments with the same law as $X_1$, we have
\bes{
&\E\max_{1\le k\le n}\Big|\frac1{\sqrt n}\sum_{i=1}^k(X_i-G_i)\Big|^4\\
\leq&\frac{N^2}{n^2}\,\E\max_{1\le k\le N}\Big|\frac1{\sqrt N}\sum_{i=1}^k(X_i-G_i)\Big|^4\\
\leq&\frac Nn\,\frac{C[\tr(Q^{1/2})]^2}{n}(1+\log N)^4\\
\leq&\frac{C[\tr(Q^{1/2})]^2}{n}(1+\log n)^4.
}
\end{proof}

Finally, we give an example to show the near optimality of the error rate in \cref{eq:thmstrong}.

\begin{example}\label{exam:strong}
Let the dimension $d=1$.
For each $n\ge2$, put
\[
 a_n=n^{1/4},\qquad
 b_n=\left(\frac{\sqrt n}{\sqrt n+1}\right)^{1/2},
\]
and let $X$ have the symmetric distribution
\[
 \Prob(X=a_n)=\Prob(X=-a_n)=\frac1{2n},\qquad
 \Prob(X=b_n)=\Prob(X=-b_n)=\frac{1-1/n}{2}.
\]
Then
\[
 \E X=0,\qquad
 \E|X|^2=1,\qquad
 \E|X|^4=\frac{2\sqrt n}{\sqrt n+1}<2.
\]
For an independent copy $X^{\prime}$, 
\[
 Q=\E|X^{\prime}-X|^4
 =2\E|X|^4+6
 =10-\frac4{\sqrt n+1}<10.
\]

Consider any coupling of i.i.d.\ copies $X_1,\ldots,X_n$ of
$X$ and mutually independent standard Gaussian variables
$G_1,\ldots,G_n$, and write
\[
 \Delta_n=\max_{1\le k\le n}
 \left|\sum_{i=1}^k(X_i-G_i)\right|.
\]
As observed in
\cite[Theorem~3 and its proof, equation~(88)]{Sakhanenko2006}, we have
\[
 \max_{1\le i\le n}|X_i|
 \le2\Delta_n+\max_{1\le i\le n}|G_i|.
\]
It follows that
\begin{align*}
 \Prob(\Delta_n\ge a_n/4)
 &\ge1-(1-1/n)^n
 -\Prob\left(\max_{1\le i\le n}|G_i|>a_n/2\right)\\
 &\ge1-e^{-1}-2n\exp(-\sqrt n/8)
 \ge\frac12
\end{align*}
for all sufficiently large $n$. Consequently, every such
coupling satisfies
\begin{equation*}
 \Big(\E\max_{1\le k\le n}
 \Big|\frac1{\sqrt n}\sum_{i=1}^k(X_i-G_i)\Big|^4\Big)^{1/2}\ge \E\max_{1\le k\le n}
 \Big|\frac1{\sqrt n}\sum_{i=1}^k(X_i-G_i)\Big|^2
 \ge\frac1{32\sqrt n}
 \ge\frac{\tr(Q^{1/2})}{32\sqrt{10}\sqrt n}.
\end{equation*}
This proves the optimality of the error rate up to the
logarithmic factor in \cref{eq:thmstrong}.  
\end{example}

\begingroup
\interlinepenalty=10000
\raggedright
\bibliographystyle{plainnat}

\endgroup
\end{document}